\documentclass[11pt,reqno]{amsart}
\usepackage[a4paper,textwidth=155mm,textheight=235mm,centering]{geometry}
\usepackage{lmodern}
\usepackage{amsmath,amssymb,amsthm}
\usepackage{microtype}
\usepackage[hidelinks]{hyperref}
\newtheorem{theorem}{Theorem}[section]
\newtheorem{lemma}[theorem]{Lemma}
\newtheorem{proposition}[theorem]{Proposition}

\newtheorem{definition}[theorem]{Definition}
\newtheorem{definitionn}[theorem]{Definitions}

\numberwithin{equation}{section}
\usepackage{xcolor}

\title[Bounded-orbit lattice representations]{Bounded-orbit lattice representations of finite groups}

\author[J.-L. Du]{Jia-Li Du}
\address{School of Mathematical Sciences, Nanjing Normal University, Nanjing 210023, P.R. China; Ministry of Education Key Laboratory of NSLSCS, Nanjing 210023, P.R. China}
\email{dujl@njnu.edu.cn}

\author[A. Lucchini]{Andrea Lucchini}
\address{Dipartimento di Matematica ``Tullio Levi-Civita'', University of Padova, Via Trieste 63, 35121 Padova, Italy}
\email{lucchini@math.unipd.it}

\author[J. Morris]{Joy Morris}
\address{Department of Mathematics and Computer Science, University of Lethbridge, 4401 University Drive, Lethbridge, AB T1K 3M4, Canada}
\email{joy.morris@uleth.ca}
\thanks{The first author was supported by the National Natural Science Foundation of China (12571370). The third author was supported by the Natural Science and Engineering Research Council of Canada (grant RGPIN-2024-04013).}

\author[P. Spiga]{Pablo Spiga}
\address{Dipartimento di Matematica e Applicazioni, University of Milano-Bicocca, Via Cozzi 55, 20125 Milano, Italy}
\email{pablo.spiga@unimib.it}

\date{}

\begin{document}

\begin{abstract}
For a finite group $G$, let $\lambda(G)$ denote the minimum number of orbits on the elements of a finite lattice $L$ with $\operatorname{Aut}(L)\cong G$. Babai and Goodman conjectured that $\lambda(G)$ is bounded by an absolute constant. We prove that $\lambda(G)\leq 50$ for every finite group $G$, thereby confirming their conjecture.
Moreover, the lattice can be chosen to have a regular orbit. The main algebraic ingredient is a decomposition of a generating set of an arbitrary finite $2$-group into an elementary abelian part and two sets in which no quotient of distinct elements is an involution.

\end{abstract}

\maketitle

\section{Introduction}

The problem of realizing an abstract group as the full automorphism group of a combinatorial structure has a long history. Frucht~\cite{Frucht1938} proved that every finite group is the automorphism group of a finite graph. The corresponding existence theorem for finite lattices goes back to Birkhoff~\cite{Birkhoff}; see also Frucht~\cite{Frucht1950}. Once existence is established, one can ask how small or how symmetric a representing structure can be. The number of elements determines its size, whereas the number of automorphism orbits  provides an indication of its degree of symmetry. This is the point of view taken by Babai.

For graphs these two questions lead to rather different results. Babai~\cite{Babai1974} showed that, apart from the cyclic groups of orders $3$, $4$ and $5$, every finite group has a graph representation with at most two vertex-orbits. In contrast, Goodman~\cite{Goodman} proved that there is no absolute bound on the number of edge-orbits needed to represent all finite groups. Babai, Goodman and Lov\'asz~\cite{BGL} developed the framework for this edge-orbit problem, relating graph representations to subgroup partitions and distinguished elements of the group. The present paper builds on their work.

The corresponding questions for posets have been answered affirmatively in~\cite{MS}: for every finite group $G$, there exists a poset having automorphism group isomorphic to $G$ and with at most $4$ orbits.

In this paper we are interested in lattices. A \textit{\textbf{lattice}} is a partially ordered set in which every pair $x,y$ has a greatest lower bound $x\wedge y$ and a least upper bound $x\vee y$, called its \textit{\textbf{meet}} and \textit{\textbf{join}}, respectively. We say that a finite lattice $L$ \textit{\textbf{represents}} a finite group $G$ if $\operatorname{Aut}(L)\cong G$, where automorphisms preserve the order.  Write
\[
\lambda(G)=\min\bigl\{|L/\operatorname{Aut}(L)|:
 L\text{ is a finite lattice representing }G\bigr\}.
\]
Here $L/\operatorname{Aut}(L)$ denotes the set of orbits on the elements of $L$. An orbit is \textit{\textbf{regular}} if the stabilizer of each of its points is trivial.

Babai~\cite[Conjecture~5.29]{Babai1981}  originally conjectured that no absolute bound exists on the number of orbits that suffices to represent every finite group as the automorphism group of a finite lattice. In our notation,
this says that $\lambda(G)$ is unbounded as $G$ ranges over the finite
groups. Babai and Goodman~\cite[Corollary~4.12]{BG} later proved that every
finite group $G$ has a lattice representation with at most $c|G|$
elements, where $c$ is an absolute constant. In the same paper, they
proposed the opposite of Babai's original conjecture: there is an
absolute constant $C$ such that $\lambda(G)\leq C$ for every finite
group $G$~\cite[Conjecture~4.13]{BG}. This would strengthen their result
on the number of elements, since every orbit has at most $|G|$ elements,
and hence a lattice representation with at most $C$ orbits has at most
$C|G|$ elements.

The standard construction, introduced by Birkhoff~\cite{Birkhoff}, is quite elementary. Given a finite group $G$, it starts with a finite graph $\Gamma$ with $\mathrm{Aut}(\Gamma)\cong G$ and it transforms $\Gamma$ into a lattice $L$ by considering the vertex--edge incidence determined
 by $\Gamma$. It turns out that $\operatorname{Aut}(L)\cong\operatorname{Aut}(\Gamma)$ and
$|L/\operatorname{Aut}(L)|
 =2+|V(\Gamma)/\operatorname{Aut}(\Gamma)|
    +|E(\Gamma)/\operatorname{Aut}(\Gamma)|$;
see \cite{Frucht1950} and \cite[Corollary~4.12]{BG}. 
Thus boundedly many edge-orbits suffice for boundedly many lattice orbits. Goodman's theorem shows, however, that this construction cannot settle the lattice conjecture for all finite groups. 

The connection with subgroup partitions also motivated our recent work in~\cite{DMS}.  Our discussion there of the lattice conjecture emphasizes the need to combine subgroup information (like the main result in~\cite{DMS}) with further relational data. We pursue that approach here.
Our main result proves the bounded-orbit conjecture of Babai-Goodman.

\begin{theorem}\label{thm:main}
For every finite group $G$, there is a finite lattice $L$ such that
$\operatorname{Aut}(L)\cong G$ and 
$|L/\operatorname{Aut}(L)|\leq50$.
The lattice can be chosen to have a regular orbit. 
\end{theorem}

It is likely that this bound is not optimal. Even within our approach
to the construction, in many (perhaps all) situations some orbits
could be omitted while still obtaining a lattice representing the
same group. A different construction might yield a substantially
smaller universal bound. However, four orbits do not suffice in
general: the cyclic group of order three cannot be represented by
a lattice with at most four orbits, even without requiring a regular
orbit. We do not know the best possible universal bound, or whether
requiring a regular orbit changes that bound.

The main algebraic step in our proof is Theorem~\ref{thm:five}. For every finite $2$-group $P$, it gives a generating set
$X\mathbin{\dot\cup}T_1\mathbin{\dot\cup}T_2$
whose images form a basis of $P/\Phi(P)$, such that $\langle X\rangle$ is elementary abelian and
$(st^{-1})^2\neq1$ for every  $s,t\in T_i$ with $s\neq t$.

We conclude this introductory section by observing that in Theorem~\ref{thm:main} one cannot replace ``lattice'' by ``distributive lattice'', as proved by Babai in~\cite[Proposition~4.8]{Babai0}.

\section{Finite \texorpdfstring{$p$}{p}-groups and quadratic forms}\label{sec:p-prelim}

All groups, vector spaces and partially ordered sets considered below are finite. Throughout this section, $P$ denotes a finite $p$-group. We use multiplicative notation for groups, with identity $1$, and additive notation for vector spaces. 
Write $P'$ for the derived subgroup, $P^p=\langle x^p:x\in P\rangle$, and $d(P)$ for the least number of generators of $P$. The \textit{\textbf{Frattini subgroup}} $\Phi(P)$ is the intersection of all maximal proper subgroups, with $\Phi(1)=1$. A finite abelian group of exponent dividing $p$ is \textit{\textbf{elementary abelian}}; we regard it as a vector space over the finite field $\mathbb F_p$ of cardinality $p$.

We use the standard facts about finite $p$-groups collected in Robinson~\cite[Sections~5.2--5.3]{Rob}. In particular, every minimal nontrivial normal subgroup of a finite $p$-group is central of order $p$, and the Burnside basis theorem gives
\begin{equation}\label{eq:frattini}
\Phi(P)=P'P^p,
\qquad
 d(P)=\dim_{\mathbb F_p}P/\Phi(P).
\end{equation}
A subset of $P$ generates $P$ if and only if its images under the natural projection of $P$ onto $P/\Phi(P)$ span $P/\Phi(P)$. We also use the immediate consequences
\begin{equation}\label{eq:frattini-functorial}
\Phi(H)\leq\Phi(P)\quad(\hbox{for every } H\leq P),
\qquad
\Phi(P/N)=\Phi(P)N/N\quad(\hbox{for every } N\trianglelefteq P).
\end{equation}
The next two lemmas record the particular consequences needed in the proof of Theorem~\ref{thm:five}.

\begin{lemma}\label{lem:independent-lifts}
Suppose $x_1,\ldots,x_m\in P$ have independent images in $P/\Phi(P)$, and put $H=\langle x_1,\ldots,x_m\rangle$. Then $d(H)=m$.
\end{lemma}
\begin{proof}
By \eqref{eq:frattini-functorial}, $\Phi(H)\leq\Phi(P)$, so the mapping $H/\Phi(H)\to P/\Phi(P)$ is defined. Its image has dimension $m$. Its domain has dimension $d(H)\leq m$, since the $x_i$ generate $H$. A surjection onto an $m$-dimensional image from a space of dimension at most $m$ is injective and has domain of dimension $m$. 
\end{proof}
We collect here some standard definitions from linear algebra.
\begin{definitionn}{\rm
A map $Q:V\to\mathbb F_2$, where $V$ is an $\mathbb F_2$-vector space, is a \textit{\textbf{quadratic form}} if $Q(0)=0$ and
$b(u,v)=Q(u+v)+Q(u)+Q(v)$
is bilinear. The form $b$ is called its \textit{\textbf{polar form}}. It is \textit{\textbf{alternating}}, that is, $b(v,v)=0$ for every $v$. Its \textit{\textbf{radical}} is
$\operatorname{rad}(b)=\{v\in V:b(v,w)=0\text{ for all }w\in V\}$.
We call $b$, and the quadratic space with polar form $b$, \textit{\textbf{nondegenerate}} when $\operatorname{rad}(b)=0$. Two subspaces $U,W$ are \textit{\textbf{orthogonal}} if $b(U,W)=0$. We write $V=U\perp W$ when $V=U\oplus W$ and $U,W$ are orthogonal.
The restriction of $Q$ to $\operatorname{rad}(b)$ is linear. Thus $
K=\{v\in\operatorname{rad}(b):Q(v)=0\}$
is a subspace.

A vector is \textit{\textbf{singular}} if its quadratic value is zero and \textit{\textbf{nonsingular}} if its quadratic value is one. A nondegenerate quadratic plane is \textit{\textbf{hyperbolic}} if it has a basis $e,f$ with $Q(e)=Q(f)=0$, $b(e,f)=1$, and is \textit{\textbf{anisotropic}} if all its nonzero vectors have quadratic value one.
}
\end{definitionn}

\begin{lemma}\label{lem:quadratic-section}
Let $P$ be a finite $2$-group with $\Phi(P)\neq1$. There is a subgroup $N\trianglelefteq P$, with $N<\Phi(P)$, such that $\Phi(P)/N$ is central of order two in $P/N$. Upon identifying it with the additive group of $\mathbb F_2$, the rule
\begin{equation}\label{eq:Q-section}
Q(x\Phi(P))=x^2N
\end{equation}
defines a nonzero quadratic form on $V=P/\Phi(P)$.
\end{lemma}
\begin{proof}
As $\Phi(P)\ne 1$, among the proper subgroups of $\Phi(P)$ that are normal in $P$,  we can choose one maximal by inclusion and call it $N$.  Hence $\Phi(P)/N$ is minimal nontrivial normal in $P/N$ and hence $\Phi(P)/N$ is central of order two.

Put $\overline P=P/N$ and $Z_0=\Phi(P)/N$. Equation~\eqref{eq:frattini-functorial} gives $\Phi(\overline P)=Z_0$.
All squares and all commutators of $\overline P$ lie in $Z_0$. In particular, $\overline P$ has class at most two and exponent at most four. If $z\in Z_0$, then $z$ is central and $z^2=1$, whence $(xz)^2=x^2$. This proves that \eqref{eq:Q-section} is independent of the representative modulo $\Phi(P)$.

In a group with central commutators one has
$(xy)^2=x^2y^2[y,x]$.
Indeed, the identity $yx=xy[y,x]$ permits the middle two factors in $xyxy$ to be interchanged, and the commutator is central. Also
$[xy,z]=[x,z][y,z]$ and $[x,yz]=[x,y][x,z]$,
when the commutators are central. In $\overline P$, the commutators have order at most two. Therefore the polar form of $Q$ is the map induced by commutation, and the two displayed commutator identities show that it is bilinear. It factors through $V$, because $Z_0$ is central.

Finally, if $Q=0$, every element of $\overline P$ has square one and hence it is elementary abelian. Thus \eqref{eq:frattini} gives $\Phi(\overline P)=1$, contrary to $|Z_0|=2$. Thus $Q\neq0$.
\end{proof}

\section{The scalar quadratic decomposition}\label{sec:quadratic}

We now prove the linear-algebra statements that we will use to decompose a generating set. The first result is simply an observation, and we include a proof for completeness.

\begin{lemma}\label{lem:alternating}
Let $b$ be an alternating bilinear form on a finite-dimensional space over $\mathbb F_2$, and let $R=\operatorname{rad}(b)$.
\begin{enumerate}
\item Every vector-space complement $U$ to $R$ is nondegenerate for $b$.
\item Every nondegenerate alternating space is an orthogonal direct sum of nondegenerate two-dimensional spaces. In particular, its dimension is even, and all nondegenerate alternating forms of the same dimension are equivalent.
\item In even nondegenerate dimension $r$, there is a basis $w_1,\ldots,w_r$ with $b(w_i,w_j)=1$ whenever $i\neq j$.
\end{enumerate}
\end{lemma}
\begin{proof}
If $u\in U$ is perpendicular to $U$, it is also perpendicular to $R$, and hence to $V=U\oplus R$. Thus $u\in U\cap R=0$, proving (1).

For (2), if the space is nonzero, choose $e\neq0$ and then $f$ with $b(e,f)=1$. Their span is nondegenerate. Every vector can be uniquely expressed as a vector in this plane plus a vector perpendicular to it, and the perpendicular complement is again nondegenerate. Induction splits the space into planes such that the bilinear form restricted to these planes has associated matrix $\left(\begin{smallmatrix}0&1\\1&0\end{smallmatrix}\right)$, proving even dimension and equivalence.

For (3), consider the $r\times r$ matrix with zero diagonal and all off-diagonal entries one. It is alternating. If it annihilates a column $x$, and $t$ is the sum of the coordinates of $x$, its $i$th equation is $t+x_i=0$. All $x_i$ therefore equal $t$, and $t=rt=0$ since $r$ is even. Thus the matrix is nondegenerate. Part (2) identifies it with the matrix of the given form in a suitable basis. The case $r=0$ is interpreted as an empty basis.
\end{proof}

\begin{lemma}\label{lem:planes}
Let $Q:V\to\mathbb{F}_2$ be a quadratic form over the field with 2 elements  having nondegenerate polar form. 
\begin{enumerate}
\item In dimension two, the space is either hyperbolic or anisotropic.
\item In every dimension at least four, the space contains both a hyperbolic plane and an anisotropic plane.
\end{enumerate}
\end{lemma}
\begin{proof}
If a nondegenerate plane has a nonzero singular vector $e$, choose $w$ with $b(e,w)=1$ and set $f=w+Q(w)e$. Then
$Q(f)=Q(w)+Q(w)b(w,e)=0$ and $b(e,f)=1$.
The vectors $e,f$ give a hyperbolic basis. If no nonzero vector is singular, the plane is anisotropic. This proves (1).

For (2), split off two orthogonal nondegenerate planes using Lemma~\ref{lem:alternating}. Each contains a nonsingular vector: for a pair $(a,b)$ with polar product one, the three values $Q(a),Q(b),Q(a+b)$ cannot all vanish. Choose one nonsingular vector in each plane. Their sum $e$ is nonzero and singular, since the planes are orthogonal. Choose $w$ with $b(e,w)=1$ and form $f=w+Q(w)e$ as above. This gives a hyperbolic plane $\langle e,f\rangle$.

Its perpendicular complement is nondegenerate and has dimension at least two, so it contains a vector $z$ with $Q(z)=1$. The vectors $e+f$ and $e+z$ satisfy
$Q(e+f)=Q(e+z)=1$ and $b(e+f,e+z)=1$.
Their sum also has quadratic value one. They therefore span an anisotropic plane.

\end{proof}

\begin{lemma}\label{lem:scalar}
Let $Q:V\to\mathbb F_2$ be a nonzero quadratic form, let $b$ be its polar form, and put $K=\operatorname{rad}(b)\cap Q^{-1}(0)$.
Then one of the following holds:
\begin{enumerate}
\item there are two (possibly empty)  disjoint sets $D_1,D_2$ whose union is a basis of a complement to $K$, such that
\begin{equation}\label{eq:local-five}
Q(v)=1\quad \forall v\in D_i,
\qquad b(v,w)=1\quad \forall v,w\in D_i,\ v\neq w,
\end{equation}
\item there are $u,v\in V$ such that
\begin{equation}\label{eq:hyperbolic}
V=K\oplus\langle u,v\rangle,
\qquad Q(u)=Q(v)=0,
\qquad b(u,v)=1.
\end{equation}
\end{enumerate}
Furthermore, in conclusion (1), $Q(x+k)=1$ for every $x\in D_1\cup D_2$ and every $k\in K$. In conclusion (2), $Q(u+v+k)=1$ for every $k\in K$.
\end{lemma}
\begin{proof}
The last assertions follow from $Q(k)=0$ and $b(k,V)=0$. Now, we construct the complementary vectors.

We first deal with the case $Q$ is nonzero on $R=\operatorname{rad}(b)$.
Choose $z\in R$ with $Q(z)=1$. Since $Q|_R$ is linear, $R=K\oplus\langle z\rangle$. Choose a complement $U_0$ to $R$. By Lemma~\ref{lem:alternating}, it is nondegenerate, has even dimension $r$, and has a basis $w_1,\ldots,w_r$ with pairwise polar products one. Put
\[
v_i=w_i+(1+Q(w_i))z.
\]
The vector $z$ is in the polar radical, so $Q(v_i)=1$, and $b(v_i,v_j)=b(w_i,w_j)=1$ for $i\neq j$. The vectors $v_1,\ldots,v_r,z$ form a complementary basis to $K$: projecting onto $U_0$ shows independence of the $v_i$, and the remaining radical direction is $z$. Put the $v_i$ in one set and $z$ in a second. If $r=0$, use only $\{z\}$. Padding with empty sets gives (1).

Assume now that $Q$ vanishes on $R$.
Here $K=R$. Choose a nondegenerate complement $U$ of even dimension $r$. Since $Q\neq0$, we have $r\geq2$. For $r=2$, Lemma~\ref{lem:planes}(1) gives either a hyperbolic plane and hence conclusion (2) holds, or an anisotropic plane. In the anisotropic case any basis consists of two vectors of quadratic value one and polar product one; putting them in one set gives (1).

Suppose now that $r\geq4$. We first construct $r-2$ independent vectors $v_1,\ldots,v_{r-2}\in U$ such that
\[
Q(v_i)=1 \hbox{ for every }i,
\qquad b(v_i,v_j)=1\quad \hbox{ for every }i\neq j.
\]
The construction proceeds in pairs. Assume an even number $m$ have been constructed, with $m<r-2$, and put $a=v_1+\cdots+v_m$; for $m=0$, put $a=0$. Their span $W_m$ is nondegenerate, because its polar matrix is the even-sized matrix in Lemma~\ref{lem:alternating}(3). Hence $U=W_m\perp W_m^\perp$, and
\[
\dim W_m^\perp=r-m\geq4.
\]
By Lemma~\ref{lem:planes}(2), this complement contains a pair $u,v$ with
\[
Q(u)=Q(v)=1+Q(a),
\qquad b(u,v)=1:
\]
use an anisotropic basis if $Q(a)=0$, and a hyperbolic basis if $Q(a)=1$. Append $a+u$ and $a+v$.
Their quadratic values are $Q(a)+Q(u)=1$ and $Q(a)+Q(v)=1$, respectively, since $u,v\perp W_m$. Their mutual polar product is one. For any earlier $v_i$,
\[
b(a+u,v_i)=b(a+v,v_i)=\sum_{j=1}^{m}b(v_j,v_i)=m-1=1
\]
in $\mathbb F_2$; there are no such conditions when $m=0$. The new vectors are independent of the preceding ones because their projections onto $W_m^\perp$ are the independent vectors $u,v$. Thus the induction continues.

After $r-2$ vectors have been selected, put $D_1=\{v_1,\ldots,v_{r-2}\}$ and $W=\langle D_1\rangle$.
The space $W$ is nondegenerate, and $W^\perp$ is a nondegenerate plane. If this plane is anisotropic,
use any basis of it as $D_2$. If it is hyperbolic, choose a basis $e,f$ with $Q(e) = Q(f) = 0$ and
$b(e, f) = 1$, put $w = v_1$, and use 
$D_2 = \{e+w,f+w\}$. Since $w\perp e,f$
and $Q(w)=1$, we have 
\[
Q(e+w)=Q(f+w)=1,
\qquad b(e+w,f+w)=1.
\]
The images of these two vectors modulo $W$ are the basis $e,f$ of $W^\perp$. Thus in either case
$D_1\cup D_2$ is a basis of $U$, and all the conditions in (1) hold.
\end{proof}

\section{Two sets for an arbitrary finite \texorpdfstring{$2$}{2}-group}\label{sec:five}

Let $P$ be a group. An \textit{\textbf{involution}} is an element of order two. Call a set $T\subseteq P$ \textit{\textbf{admissible}} if
\begin{equation}\label{eq:admissible}
(st^{-1})^2\neq1\qquad(\forall s,t\in T,\ s\neq t).
\end{equation}
The empty set and singletons are admissible. The condition is symmetric in $s,t$, because reversing them inverts the quotient. The following result is the key ingredient for solving Babai and Goodman's conjecture on lattices.

\begin{theorem}\label{thm:five}
Let $P$ be a finite $2$-group. There exist a set $X\subseteq P$ and two possibly empty sets $T_1,T_2\subseteq P$ such that:
\begin{enumerate}
\item $A=\langle X\rangle$ is elementary abelian;
\item the images of the disjoint union $X\mathbin{\dot\cup}T_1\mathbin{\dot\cup}T_2$ 
form a basis of $P/\Phi(P)$;
\item every $T_i$ is admissible.
\end{enumerate}
In particular, the images of $X$ form a basis of $A\Phi(P)/\Phi(P)$, and $P=\langle A,T_1,T_2\rangle$.
\end{theorem}
\begin{proof}
Use induction on $d(P)$. For $P=1$, all the sets are empty. If $\Phi(P)=1$, then~\eqref{eq:frattini} makes $P$ elementary abelian. Take $X$ to be a basis, and take all $T_i$ empty. 

Suppose $\Phi(P)\neq1$. Choose $N$ and the nonzero quadratic form $Q$ from Lemma~\ref{lem:quadratic-section}, on $V=P/\Phi(P)$. Let $b$ be its polar form and put $K=\operatorname{rad}(b)\cap Q^{-1}(0)$. Because $Q\neq0$, the subspace $K$ is proper in $V$.

Choose lifts of a basis of $K$, and let $H$ be the subgroup they generate. Lemma~\ref{lem:independent-lifts} gives
$d(H)=\dim K<\dim V=d(P)$. Apply induction to $H$, obtaining $A=\langle X\rangle$ elementary abelian and two sets $T_1,T_2$. Their images, together with those of $X$, are a basis of $K$ when viewed inside $V$. Throughout the following argument, a bar denotes image in $V$.

Assume first that conclusion (1) of Lemma~\ref{lem:scalar} holds.
Choose lifts $L_i\subseteq P$ of the vectors in $D_i$, and replace each $T_i$ by $T_i\cup L_i$. If $x,y\in L_i$ are distinct, then
\[
Q(\bar x+\bar y)=Q(\bar x)+Q(\bar y)+b(\bar x,\bar y)=1.
\]
Since $\overline{xy^{-1}}=\bar x+\bar y$ in the elementary abelian quotient $V$, equation~\eqref{eq:Q-section} gives $(xy^{-1})^2N \neq N$. Thus $(xy^{-1})^2\neq1$ in $P$.

If $x\in L_i$ and $t\in T_i$, then $\bar t\in K$, so Lemma~\ref{lem:scalar} gives $Q(\bar x+\bar t)=1$. The same argument shows $(xt^{-1})^2\neq1$; the reversed quotient satisfies the same condition. Pairs already in $T_i$ remain admissible by induction. Thus every enlarged set is admissible. The union of the $D_i$ is a complementary basis to $K$, so the union $X\cup T_1\cup T_2$ remains disjoint and has basis images in $V$. The set $X$ and subgroup $A$ have not changed. 

Assume next that conclusion (2) of Lemma~\ref{lem:scalar} holds.
Choose lifts $x,y\in P$ of the vectors $u,v$ in \eqref{eq:hyperbolic}. Keep $X$ and $A$ unchanged, and define
\begin{align*}
T_1^*&=\{x\}\cup\{ty:t\in T_1\},\notag\\
T_2^*&=\{y\}\cup\{tx:t\in T_2\}.
\end{align*}
Here the old sets may be empty, in which case the corresponding new sets are just the indicated singletons.

Right multiplication leaves internal quotients unchanged,
consequently distinct translated members of the same set satisfy \eqref{eq:admissible}. For the new pairs involving $x$ or $y$, use $\bar t\in K$ to obtain
\[
Q\bigl(\overline{x(ty)^{-1}}\bigr)
 =Q(u+v+\bar t)=1
\qquad (\forall t\in T_1).
\]
Thus $\bigl(x(ty)^{-1}\bigr)^2\notin N$, and in particular it is not $1$. Similarly,
\[
Q\bigl(\overline{y(tx)^{-1}}\bigr)=Q(v+u+\bar t)=1
\qquad(\forall t\in T_2).
\]

As above, it is clear that $X\cup T_1^*\cup T_2^*$ is a disjoint union and that their images in $V$ form a basis.

Both alternatives therefore preserve the induction statement. The resulting basis images generate $P$ by the Burnside basis theorem, which also proves the last assertion of the theorem.
\end{proof}

\section{Relational data with a prescribed permutation group}\label{sec:codes}

For a finite group $G$ and $g\in G$, let $\rho_g$ denote the permutation $x\mapsto xg$. Write
\[
R(G)=\{\rho_g:g\in G\}.
\]
This is the \textit{\textbf{right regular group}} of $G$. With the usual composition of functions, the map $g\mapsto\rho_{g^{-1}}$ is an isomorphism $G\to R(G)$.

\begin{definitionn}\label{def:data}{\rm
For subsets $S_1,\ldots,S_r\subseteq G$ and subgroups $H_1,\ldots,H_s\leq G$, let $F$ be the set of all permutations $f$ of $G$ satisfying
\begin{align}\label{eq:data}
f(S_i g)&=S_i f(g) &&(\forall 1\leq i\leq r,\ \forall g\in G),\notag\\
f(H_j g)&=H_j f(g) &&(\forall 1\leq j\leq s,\ \forall g\in G).
\end{align}
These will be called the \textit{\textbf{subset relations}} and \textit{\textbf{subgroup partitions}}, respectively. They \textit{\textbf{encode}} $R(G)$ if $F=R(G)$.}
\end{definitionn}

The set $F$ is a group: the conditions are preserved by composition, and replacing $g$ by $f^{-1}(g)$ shows they are preserved by inversion. Every right translation belongs to $F$. For a subgroup $H$, the second condition means precisely that $f$ preserves the partition into right cosets of $H$: the image of the coset containing $g$ must be the coset containing $f(g)$. For a singleton $S=\{a\}$, the first condition is the identity $f(ag)=af(g)$ for every $g$.

\begin{lemma}[{\cite[Proposition~3.6]{BGL}}]\label{lem:regular}
Let $F\leq\operatorname{Sym}(G)$ contain $R(G)$. If every element of the stabilizer $F_1$ fixes every member of a generating set of $G$, then $F=R(G)$.
\end{lemma}

The following result is an elementary observation and we omit its proof.
\begin{lemma}\label{lem:partition-join}
If a permutation of $G$ preserves the right-coset partitions for $H_1,\ldots,H_s$, it also preserves the right-coset partition for $H=\langle H_1,\ldots,H_s\rangle$. In particular, if it fixes $1$, it fixes $H$ setwise.
\end{lemma}

\begin{definition}{\rm
A subset $S$ of a group $G$ is called \textit{\textbf{Sidon}} if the map
\[
\{(s,t)\in S\times S:s\neq t\}\longrightarrow G \quad\hbox{ defined by }
\quad (s,t)\longmapsto st^{-1}
\]
is injective.  Empty sets and singletons satisfy this condition.}
\end{definition}

\begin{lemma}\label{lem:sidon}
Let $P$ be a finite $2$-group. If the images of $S\subseteq P$ are independent in $P/\Phi(P)$ and $S$ is admissible, then $S$ is Sidon.
\end{lemma}
\begin{proof}
Suppose $st^{-1}=uv^{-1}$, with $s\neq t$ and $u\neq v$. In $P/\Phi(P)$ this gives
$\bar s+\bar t=\bar u+\bar v$.
Independence says that the two unordered pairs are equal. If the ordered pairs are not equal, they are reversed, so
$st^{-1}=ts^{-1}=(st^{-1})^{-1}$.
This contradicts $(st^{-1})^2\neq1$. Hence the ordered pairs are equal.
\end{proof}

The proof of the following lemma can be deduced from proofs of the elementary abelian case in~\cite[Lemma~6.3(b) and Theorem~6.4]{BGL}, we include a proof for completeness.
\begin{lemma}\label{lem:abelian-code}
Let $V\neq0$ be an $\mathbb F_2$-vector space. Its translation group is encoded by four subspace partitions. The first two subspaces can be required to span $V$.
\end{lemma}
\begin{proof}
 Choose a basis $e_1,\ldots,e_d$ and put
\begin{align*}
J_1&=\langle e_1,e_3,e_5,\ldots\rangle,&
J_2&=\langle e_2,e_4,e_6,\ldots\rangle,\\
J_3&=\langle e_1+e_2,e_3+e_4,\ldots\rangle,&
J_4&=\langle e_2+e_3,e_4+e_5,\ldots\rangle.
\end{align*}
Only pairs with both indices at most $d$ are included. 

Clearly $J_1+J_2=V$ and $J_1\cap J_2=0$. If a vector of $J_3$ belongs to $J_2$, all its odd coordinates vanish, forcing the coefficients of every displayed generator of $J_3$ to vanish. Thus $J_2\cap J_3=0$. The identical argument with even coordinates gives $J_1\cap J_4=0$.

As in Definitions~\ref{def:data}, let $F$ be the set of all permutations of $V$ preserving the subgroup partitions on $J_1,J_2,J_3,J_4$. Let $f\in F$ fix $0$, so the four subspaces are setwise fixed by $f$. We first show that $e_1$ is fixed by $f$. If $d=1$, this is immediate, because
$J_1 = \langle e_1\rangle$. If $d\geq 2$,  put
\[
W=\langle e_1,\ldots,e_{d-1}\rangle
=
\begin{cases}
	J_1+ J_4, & \text{\(d\) even},\\
	J_2+ J_3, & \text{\(d\) odd}.
\end{cases}
\]

By Lemma~\ref{lem:partition-join}, $f$ preserves the coset partition for $W$ and fixes $W$ setwise. On $W$, it preserves
the coset partitions for $J_i \cap W$, since these cosets are intersections of the original cosets with $W$. The four intersections are exactly the same construction on the basis $e_1,\cdots,e_{d-1}$.
Repeating this reduction shows that $f$ fixes $\langle e_1\rangle=\{0,e_1\}$ setwise. Since $0$ is fixed by $f$, 
so is $e_1$.

Suppose $f$ fixes $e_i$, where $i<d$ is odd. The intersection
$J_3\cap(e_i+J_2)$
contains $e_i+e_{i+1}$ and is a singleton because $J_3\cap J_2=0$. Both intersected sets are invariant under $f$, so $e_i+e_{i+1}$ is fixed. Next
$
J_2\cap(e_i+e_{i+1}+J_1)=\{e_{i+1}\}$,
and the same reasoning fixes $e_{i+1}$. When $i$ is even, use $J_4\cap(e_i+J_1)$ first, and $J_1\cap(e_i+e_{i+1}+J_2)$ second. The intersections are again singletons by the equalities already proved.

By induction the entire basis is fixed. Applying Lemma~\ref{lem:regular} to  $F$, which contains all translations, gives $F=R(V)$.
\end{proof}

\begin{proposition}\label{prop:two-code}
For every nontrivial finite $2$-group $P$, the regular group $R(P)$ is encoded by at most two nonempty Sidon subsets avoiding $1$, and at most five nontrivial subgroup partitions. The subgroups defining the partitions can be chosen to generate $P$, and at least one subset relation is present.
\end{proposition}
\begin{proof}
Let $\pi:P\to V=P/\Phi(P)$ be the natural projection. By \eqref{eq:frattini} and the Burnside basis theorem, $V\neq0$. Apply Lemma~\ref{lem:abelian-code} to $V$, with subspaces $J_1,\ldots,J_4$, and put $K_i=\pi^{-1}(J_i)$. In $P$, impose the  four subgroup partitions for
$K_1,\cdots, K_4$.

Choose $A,X,T_1,T_2$ from Theorem~\ref{thm:five}. If $A\neq1$, also impose the subgroup partition for $A$.
Since $A$ is elementary abelian and the images of $X$ are independent in $V$, we have $A\cap \Phi(P) = 1$.

Also consider the subset relation for every nonempty $T_i$. Each $T_i$ is Sidon by Lemma~\ref{lem:sidon}, and none contains $1$, by its independent nonzero images.  If both sets are empty, add the singleton relation for any nonidentity element of $P$. Remove identity subgroups, empty sets, and repeated data, but retain each distinct nontrivial $K_i$, even if
$K_i=P$. The subgroups defining the partitions generate $P$, since $\langle K_1, K_2\rangle=P$.

As in Definitions~\ref{def:data}, let $F$ be the group preserving all of these subgroup partitions and all of these subset relations, and take $f\in F_1$. 
Since $J_1 \cap J_2 = 0$, for every
$g \in P$ we have
\[\Phi(P)g = K_1g\cap K_2g.\]
Thus the given partitions preserve the partition into $\Phi(P)$-cosets, and $f$ induces a permutation $\bar f$ on $V$. 
The preimage partitions induce the four subspace partitions.
As $f(1)=1$, the permutation $\bar f$ fixes $0$. Lemma~\ref{lem:abelian-code} implies $\bar f=1$; thus every $\Phi(P)$-coset is fixed setwise.

At $g=1$ the relation for $T_i$ says $f(T_i)=T_i$. Distinct members of $T_i$ have distinct images in $V$. A permutation preserving each coset and the set $T_i$ must therefore fix every member of $T_i$ individually.

The partition for $A$ gives $f(A) = A$. Since $A\cap \Phi(P) = 1$,
every $\Phi(P)$-coset contains at most one element of $A$. Therefore $f$ fixes $A$ pointwise.

Thus every element of $F_1$ fixes $A$ pointwise and fixes every $T_i$ pointwise. These sets generate $P$, so Lemma~\ref{lem:regular} gives $F=R(P)$. The count is at most $4+1=5$ subgroup partitions and two subset relations.
\end{proof}

\begin{lemma}\label{lem:once}
Let $h_1,\ldots,h_d$ be an irredundant generating set of a group: no $h_i$ belongs to the subgroup generated by the others. A word in these generators and their inverses, in which some $h_i$ occurs exactly once, cannot equal $1$.
\end{lemma}
\begin{proof}
Write the word as $u h_i^{\varepsilon}v$, with $\varepsilon\in\{1,-1\}$, where $u,v$ use only the other generators. If it were $1$, then $h_i^{\varepsilon}=u^{-1}v^{-1}$, expressing $h_i$ in those other generators, a contradiction. 
\end{proof}

\begin{proposition}\label{prop:odd-code}
If $H\neq1$ has odd order, then $R(H)$ is encoded by at most two nonempty Sidon subsets avoiding $1$, with no subgroup partitions.
\end{proposition}
\begin{proof}
Choose an ordered irredundant generating set $h_1,\ldots,h_d$. Put
\[
S=\{h_1,\ldots,h_d\},\qquad
\delta_i=h_{i+1}h_i^{-1}\quad(1\leq i<d),\qquad
D=\{\delta_i:1\leq i<d\}.
\]
We use the relations for $S,D$, omitting the empty set $D$ if $d=1$ and omitting repetitions. Every $\delta_i$ is nonidentity, since adjacent generators are distinct.

We first verify the Sidon assertions. If
$h_i h_j^{-1}=h_k h_l^{-1}$, for $i\neq j$ and $k\neq l$,
then, after moving the right side to the left, Lemma~\ref{lem:once} rules out any generator index occurring just once among the four positions. There must therefore be exactly two indices, each occurring twice, and $\{i,j\}=\{k,l\}$. If the ordered pairs are not equal, they are reversed. Then $h_i h_j^{-1}$ is equal to its inverse and is nonidentity, giving an involution in $H$. Odd order excludes this. Thus $S$ is Sidon.

The elements $\delta_i$ are distinct. Indeed, if $\delta_i=\delta_j$ with $i<j$, the relation
\[
h_{i+1}h_i^{-1}h_jh_{j+1}^{-1}=1
\]
contains $h_i$ exactly once, contrary to Lemma~\ref{lem:once}.

Now consider
\begin{equation}\label{eq:D-collision}
\delta_i\delta_j^{-1}=\delta_k\delta_l^{-1},
\qquad i\neq j,\quad k\neq l.
\end{equation}
Regard each occurrence of an index $a\in\{i,j,k,l\}$ as the edge $\{a,a+1\}$ of the path on vertices $1,\ldots,d$. There are four edge occurrences, counted with multiplicity. Let $a$ be the least distinct edge index and $b$ the greatest. They are different, because $i\neq j$. In the expanded relation obtained from \eqref{eq:D-collision}, the generator $h_a$ occurs as many times as edge $a$ occurs: no edge with smaller index is present, and no larger edge contains that vertex. Similarly, $h_{b+1}$ occurs as many times as edge $b$ occurs. Lemma~\ref{lem:once} forces each of these multiplicities to be at least two. With only four occurrences available, precisely two distinct edges occur, each twice. It follows that $\{i,j\}=\{k,l\}$.

The only remaining unequal ordered-pair possibility reverses the pair. It makes $\delta_i\delta_j^{-1}$ a nonidentity element equal to its inverse, again impossible in odd order. Hence $D$ is Sidon.

Let $F$ be the group preserving these relations, and take $f\in F_1$. Then $S$ is setwise fixed by $f$. Let us prove that $f$ actually fixes every element of $S$. This is obvious if $d=1$, so assume $d>1$. For $1\leq i<d$, we claim
\begin{equation}\label{eq:successor}
S\cap Dh_i=\{h_{i+1}\}.
\end{equation}
The displayed element belongs to the intersection. Conversely, an element in it gives
$h_k=h_{j+1}h_j^{-1}h_i$,
for some $j$. By Lemma~\ref{lem:once}, there can be no index occurring once among $k,j+1,j,i$. Since $j\neq j+1$, the only possibilities are
$k=j+1$ and $i=j$, or $k=j$ and $i=j+1$.
The first gives the desired successor. The second says $h_j=h_{j+1}h_j^{-1}h_{j+1}$,
so $(h_{j+1}h_j^{-1})^2=1$, an impossibility. This proves \eqref{eq:successor}.

The same argument shows that $S\cap Dh_d =\emptyset$. Hence the directed relation on $S$ given
by $z \in Dy$ is exactly the path

\[h_1\longrightarrow h_2\longrightarrow \cdots \longrightarrow h_d.\]
Since $f$ preserves $S$ and the $D$ relation, it induces an automorphism of
this directed path. Its initial vertex is unique, and every successive vertex is then fixed. Thus
$f$ fixes all generators. Lemma~\ref{lem:regular} therefore gives $F=R(H)$.
\end{proof}

\section{Reduction to arbitrary finite groups}\label{sec:general}

 The reduction to a $2$-subgroup and an odd-order subgroup uses the following theorem.

\begin{theorem}[Breuer–Guralnick~{\cite[Corollary 1.2]{AG}}]\label{thm:AG}
Every finite group $G$ has a Sylow $2$-subgroup $P$ and a subgroup $H$ of odd order such that
\[G=\langle P,H\rangle.\]
\end{theorem}

\begin{lemma}\label{lem:combine}
Suppose $G=\langle H_1,\ldots,H_k,g_1,\ldots,g_t\rangle$, and each $R(H_i)$ has an encoding as in Definition~\ref{def:data}. Regard those subsets and subgroups as subsets and subgroups of $G$. Add the partition for each nontrivial $H_i$ and the singleton $\{g_j\}$ for every $g_j\neq1$. The resulting data encode $R(G)$. An added partition for $H_i$ may be omitted if its preservation already follows
from the component data.
\end{lemma}
\begin{proof}
Let $F\le\mathrm{Sym}(G)$ preserve all the subsets and all the subgroups, and take $f\in F_1$. The added partitions, or their preservation by the component data, give $f(H_i)=H_i$. Its restriction to $H_i$ preserves that subgroup's encoding for every subset and for every subgroup lying in $H_i$, and fixes $1$. Since the full preserving group there is $R(H_i)$, $f_{|H_i}$ is the identity. Each added singleton similarly fixes the corresponding $g_j$. Lemma~\ref{lem:regular} gives $F=R(G)$.
\end{proof}

\begin{proposition}\label{prop:general-code}
For every nontrivial finite group $G$, the group $R(G)$ is encoded by at most four nonempty Sidon subsets avoiding $1$, and at most five nontrivial subgroup partitions. 
\end{proposition}
\begin{proof}
Choose $P,H$ from Theorem~\ref{thm:AG}, so that $G=\langle P,H\rangle$.
Use Proposition~\ref{prop:two-code} for $P$ when $P\neq1$, and Proposition~\ref{prop:odd-code} for $H$ when $H\neq1$.

The subgroups in the encoding of $P$ generate $P$, so Lemma~\ref{lem:partition-join} implies preservation of the
partition into right cosets of $P$. 
For $H$, the connected components of the relation $x\longrightarrow sx$
$(s\in S)$, regarded as an undirected graph on $G$, are exactly the right cosets of $\left\langle S\right\rangle=H $. Thus
its subset relation already implies preservation of the $H$-coset partition.

 Combine the encodings using Lemma~\ref{lem:combine}. The number of subset relations is at most
$2+2=4$, and the number of subgroup partitions is at most five. At
 least one subset relation is present because $G\neq 1$. If no subgroup partition is present, add
 the partition for $G$; it imposes no restriction and still leaves at most five partitions.
\end{proof}

\section{From relational data to a finite lattice}\label{sec:lattice}
In a partially ordered set, we say that $y$ \textit{\textbf{covers}} $x$, and write $x\lessdot y$, if $x<y$ and no element lies strictly between them. In a partially ordered set with a least and a greatest element, an \textit{\textbf{atom}} is an element covering the least element, and a \textit{\textbf{coatom}} is an element covered by the greatest element.

\begin{lemma}\label{lem:incidence}
Let $\mathcal A$ and $\mathcal C$ be disjoint finite sets, and specify incidences between them. Form a partially ordered set with elements
$\{\widehat0,\widehat1\}\mathbin{\dot\cup}\mathcal A\mathbin{\dot\cup}\mathcal C$,
where $\widehat0$ is below everything, $\widehat1$ is above everything, and $a<c$ for $a\in\mathcal A$, $c\in\mathcal C$ precisely at the specified incidences. If two distinct members of $\mathcal C$ have at most one common incident member of $\mathcal A$, this poset is a lattice.
\end{lemma}
\begin{proof}
The hypothesis also implies the dual statement: two distinct members of $\mathcal A$ cannot have two common members of $\mathcal C$, since those two would have the original pair as common lower neighbours. Equivalently, the bipartite incidence graph has no cycle of length four.

For distinct $a,a'\in\mathcal A$, their meet is $\widehat0$. Their join is their unique common member of $\mathcal C$, if one exists, and is $\widehat1$ otherwise. Dually, distinct members of $\mathcal C$ have join $\widehat1$ and meet their unique common lower neighbour, or $\widehat0$. For an incident pair $a<c$, the meet is $a$ and the join is $c$. For a nonincident pair from opposite parts, the only common lower bound is $\widehat0$ and the only common upper bound is $\widehat1$. Identical elements and pairs involving an endpoint are immediate. 
\end{proof}

\begin{lemma}\label{lem:chains}
Let $L_0$ be a finite lattice. Replace every element $x$ by a nonempty finite chain $C_x$, with all the chains disjoint, and order the resulting set by the orders within the chains and by $C_x<C_y$ whenever $x<y$ in $L_0$.
This is a lattice. In particular, the assertion holds when the least and greatest elements are left as singletons.
\end{lemma}
\begin{proof}
Let $u\in C_x$ and $v\in C_y$. If $x=y$, then $u$ and $v$ are comparable within the chain $C_x=C_y$, so their join and meet are simply the greater and the lesser of the two.
If $x<y$, then $u<v$, so their meet and join are $u,v$; the reversed case is similar. Suppose $x,y$ are incomparable. Their join $z=x\vee y$ is different from both $x,y$. The least element of $C_z$ is an upper bound for $u,v$. Any other common upper bound lies in a chain $C_w$ with $w\geq x,y$, hence $w\geq z$. It is therefore above the least element of $C_z$. This proves that the latter is the join of $u,v$. Dually, their meet is the greatest element of $C_{x\wedge y}$. 
\end{proof}

\begin{theorem}\label{thm:encoding}
Let $G\neq1$ be finite. Suppose that $S_1,\ldots,S_r$ are nonempty Sidon subsets of $G$ avoiding $1$, with $r\geq1$, and $H_1,\ldots,H_s$ are nontrivial subgroups, with $s\geq 1$. Let $F$ be their full preserving permutation group as in Definition~\ref{def:data}. There is a finite lattice $L$ with
$\operatorname{Aut}(L)\cong F$.
If $F=R(G)$, then $L$ has a regular orbit and
\begin{equation}\label{eq:orbit-bound}
|L/\operatorname{Aut}(L)|
 = 3+3r+2s+\binom{r+1}{2}+\binom{s+1}{2}.
\end{equation}
\end{theorem}
\begin{proof}
We first construct a lattice in which specified types must be preserved. These type requirements will subsequently be removed by an intrinsic order construction. We start by defining elements and types. The atom types are
\begin{align*}
B&=\{B(g):g\in G\},\\
A_i&=\{A_i(g):g\in G\}\quad(1\leq i\leq r),\\
X_j&=\{X_j(H_jg):H_jg\text{ a right coset of }H_j\}\quad(1\leq j\leq s).
\end{align*}
The coatom types are
\begin{align*}
C_i&=\{C_i(g):g\in G\}\quad(1\leq i\leq r),\\
 U_i&=\{U_i(g):g\in G\}\quad(1\leq i\le r),\\
E_j&=\{E_j(g):g\in G\}\quad (1\leq j\leq s).
\end{align*} All these sets are disjoint copies, which we call \textit{\textbf{types}}. An equality of labels in different types does not identify elements. The incidences are exactly
\begin{align}\label{eq:incidences}
C_i(g)&>B(g),\ A_i(g),\notag\\
U_i(g)&>B(g),\ A_i(tg)\quad(t\in S_i),\\
E_j(g)&>B(g),\ X_j(H_jg).\notag
\end{align}
There are no other atom--coatom comparisons. Add endpoints $\widehat0,\widehat1$ as in Lemma~\ref{lem:incidence}.

Next, we show that what we have defined above is indeed a lattice.
We check every possible source of two common lower neighbours for distinct coatoms. Two of type $U_i$, say $U_i(g),U_i(h)$ with $g\neq h$, have different $B$-neighbours. If they had two common $A_i$-neighbours, there would be $a,b,c,d\in S_i$ with $a\neq b$, $c\neq d$, and
$ag=ch$ and $bg=dh$.
These identities imply $ba^{-1}=dc^{-1}$. The Sidon property gives $b=d$ and $a=c$, and hence $g=h$, a contradiction.

The coatoms $C_i(h)$ and $U_i(g)$ can share both their $B$- and $A_i$-neighbours only if $h=g$ and $h=tg$ for some $t\in S_i$. This would put $1$ in $S_i$, which is excluded. Distinct coatoms of type $C_i$ share neither of their two lower neighbours. Coatoms associated with different indices $i$ have different private $A_i$-types, and can share at most their $B$-neighbour.

A coatom of type $E_j$ and one of type $C_i$ or $U_i$ can share only a $B$-neighbour. For $j\neq k$, coatoms of types $E_j,E_k$ likewise can share only a $B$-neighbour. Finally, $E_j(g),E_j(h)$ with $g\neq h$ have different $B$-neighbours and can share only the single coset element $X_j(H_jg)$, if their cosets agree. This completes the check. Lemma~\ref{lem:incidence} shows that the poset is a lattice $L_0$.

We now study the type-preserving automorphisms.
Suppose an automorphism of $L_0$ preserves each displayed type. Its action on $B$ is a permutation $f$ of $G$. Every coatom has a unique lower neighbour of type $B$. Thus it must send
\[
C_i(g)\mapsto C_i(f(g)),\quad
U_i(g)\mapsto U_i(f(g)),\quad
E_j(g)\mapsto E_j(f(g)).
\]
The $C_i$ incidences then force $A_i(g)\mapsto A_i(f(g))$. The $U_i$ incidences say exactly
\[
f(S_i g)=S_i f(g).
\]
To see the subgroup condition in both directions, the coset vertex $X_j(H_jg)$ is below exactly the coatoms $E_j(h)$ with $h\in H_jg$. Its image is a coset vertex below $E_j(f(g))$, necessarily $X_j(H_jf(g))$. Equality of its complete upper-neighbour set gives
\[
f(H_jg)=H_jf(g).
\]
It also determines the action on all coset vertices. Hence every type-preserving automorphism determines one and only one $f\in F$.

Conversely, for $f\in F$, use these formulas on all types and fix the endpoints. The formula $X_j(H_jg)\mapsto X_j(H_jf(g))$ is well-defined: representatives in the same coset have images in the same coset by \eqref{eq:data}. It is bijective, and the formulas preserve all incidences in both directions, using the corresponding conditions for $f^{-1}$. They therefore define an automorphism of $L_0$. This proves that its type-preserving automorphism group is isomorphic to $F$, with faithful action on $B$.

Finally we remove all type labels. Replace each coatom of $C_i$ $(1 \le i \le r)$ by a chain of $i+1$ elements. Replace each atom of $X_j$ $(1 \le j \le s)$ by a chain of $j+1$ elements. Leave all other elements unchanged and
inherit the order from $L_0$. By Lemma~\ref{lem:chains}, the result is a lattice $L$.

We must recover all types from the order of $L$ alone. First we show that every type that has just been replaced by a chain can be recovered. Every original atom in $X_j$ has $|H_j|\geq2$ upper covers. Every original coatom has at least two lower covers, as is explicit in \eqref{eq:incidences} and uses $S_i\neq\emptyset$.

Since the new chains all contain at least two elements, it follows that the upper endpoints of the new atom-chains are exactly the elements that are neither atoms nor $\widehat0$ and have at least two upper covers. Dually, the lower endpoints of the new coatom-chains are exactly the elements that are neither coatoms nor $\widehat1$ and have at least two lower covers.

Write $X_j^+$ for the upper endpoints
of the chains replacing $X_j$, and $C_i^-$
for the lower endpoints of the chains replacing $C_i$. If $x\in X_j^+$, then  $[\widehat0,x]$ is a chain of $j+2$ elements; its cardinality uniquely determines $j$ and recovers the whole fibre. Similarly, for $y\in C_i^-$ the interval $[y,\widehat1]$ has $i+2$ elements and recovers $i$ and its fibre. Thus every $X_j^+$ and $C_i^-$ is preserved.

Elements of $E_j$ are uniquely determined as the upper covers of elements of $X_j^+$. Elements of $B$ are uniquely determined as the common lower covers of elements of both $C_1^-$ and $E_1$ (this is why we require $s \ge 1$): the lower covers of these two sets are $B\dot\cup A_1$
and $B\dot\cup X_1^+$, respectively. 
Now elements of $A_i$ are uniquely determined as the lower covers of elements of $C_i^-$ that are not in $B$. Finally, elements of $U_i$ are the upper covers of elements of $A_i$ that are not in $C_i^-$. This shows that every automorphism of the unlabelled lattice $L$ induces a type-preserving automorphism of $L_0$.
Between two finite chains of the same length there is a unique 
order isomorphism, so the induced automorphism
determines the original one uniquely.
Conversely, every type-preserving automorphism of $L_0$ extends by these unique chain isomorphisms. Therefore $\operatorname{Aut}(L)\cong F$.

Suppose now that $F=R(G)$. Every original type is transitive under $R(G)$: the types indexed by $g$ are regular, and a right-coset type is transitive by right translation. Type labels cannot be exchanged, and each chain level contributes exactly one orbit. The endpoints contribute two more. Therefore
\[
|L/\operatorname{Aut}(L)|
= 2+\sum_{i=1}^{r}i+2r+s+\sum_{j=1}^{s}j+r+s+1=3+3r+2s+\binom{r+1}{2}+\binom{s+1}{2}
\]
which is \eqref{eq:orbit-bound}. The type $B$ is a regular orbit, because an element fixing such a point fixes its label $g$ in the regular action of $R(G)$.
\end{proof}

\section{Proof of the main theorem}\label{sec:conclusion}

\begin{proof}[Proof of Theorem~$\ref{thm:main}$]
For the trivial group, a one-element chain has trivial automorphism group, one orbit, and a regular orbit.
For a general nontrivial finite group $G$,  Proposition~\ref{prop:general-code} supplies data satisfying Theorem~\ref{thm:encoding} with $r\leq4$, $s\leq5$. Thus Theorem~\ref{thm:encoding} yields
$\operatorname{Aut}(L)\cong R(G)\cong G$,
and

\[
|L/\operatorname{Aut}(L)|
 \leq3+12+10+\binom{5}{2}+\binom{6}{2}
 =25+10+15=50.
\]
Moreover, a regular orbit is part of the construction.
\end{proof}

\section*{Declaration on the use of generative AI}

Generative artificial intelligence tools were used to assist in developing the proof of Theorem~4.1 and in preparing the first draft of this manuscript. All AI-generated material was subsequently reviewed, verified, and revised by the authors, who take full responsibility for the correctness and content of the paper.

\end{document}